\documentclass[12pt]{amsart}
\usepackage{amscd,amssymb,amsmath,xcolor}
\usepackage[all]{xy}
\usepackage[raiselinks=false,colorlinks=true,citecolor=blue,urlcolor=blue,
linkcolor=blue,bookmarksopen=true,pdftex]{hyperref}

\newtheorem{theorem}{Theorem}[section]
\newtheorem{lemma}[theorem]{Lemma}

\newtheorem{proposition}[theorem]{Proposition}

\theoremstyle{definition}
\newtheorem{definition}[theorem]{Definition}
\newtheorem{example}[theorem]{Example}

\newtheorem*{question}{Question}

\newtheorem{remark}{Remark}

\numberwithin{equation}{section}
\newcommand{\im}{\mathrm{Im}}

\def\gcd{\mathrm{gcd}}

\newcommand{\ab}{\textrm{ab}}
\newcommand{\Rinf}{R_{\infty}}
\newcommand{\Rcinf}{R^{\chi}_{\infty}}

\begin{document}

\title{Twisted conjugacy and commensurability invariance}
\author{Parameswaran Sankaran}
\address{Chennai Mathematical Institute SIPCOT IT Park, Siruseri, Kelambakkam, 603103, India}
\email{sankaran@cmi.ac.in}
\author{Peter Wong}
\address{Department of Mathematics, Bates College, Lewiston,
ME 04240, U.S.A.}
\email{pwong@bates.edu}
\thanks{This work was initiated during the second author's visit to the Institute of Mathematical Sciences - Chennai, India, July 31 - August 17, 2016 and carried out in subsequent vists during August 23 - September 1, 2018 (at IMSc) and December 14 - 20, 2019 (at  Chennai Mathematical Institute). The second author would like to thank the IMSc  and the CMI for their hospitality and support throughout his visits.}
\begin{abstract}
A group $G$ is said to have property $R_{\infty}$ if for every automorphism $\varphi \in {\rm Aut}(G)$, the cardinality of the set of $\varphi$-twisted conjugacy classes is infinite. Many classes of groups are known to have such property. However, very few examples are known for which $R_{\infty}$ is {\it geometric}, i.e., if $G$ has property $R_{\infty}$ then any group quasi-isometric to $G$ also has property $R_{\infty}$. In this paper, we give examples of groups and conditions under which $R_{\infty}$ is preserved under commensurability. The main tool is to employ the Bieri-Neumann-Strebel invariant.
\end{abstract}
\date{\today}
\maketitle

\newcommand{\af}{\alpha}
\newcommand{\et}{\eta}
\newcommand{\ga}{\gamma}
\newcommand{\ta}{\tau}
\newcommand{\ph}{\varphi}
\newcommand{\bt}{\beta}
\newcommand{\lb}{\lambda}
\newcommand{\wh}{\widehat}
\newcommand{\sg}{\sigma}
\newcommand{\om}{\omega}
\newcommand{\cH}{\mathcal H}
\newcommand{\cF}{\mathcal F}
\newcommand{\N}{\mathcal N}
\newcommand{\R}{\mathcal R}
\newcommand{\Ga}{\Gamma}
\newcommand{\cc}{\mathcal C}
\newcommand{\bea} {\begin{eqnarray*}}
\newcommand{\beq} {\begin{equation}}
\newcommand{\bey} {\begin{eqnarray}}
\newcommand{\eea} {\end{eqnarray*}}
\newcommand{\eeq} {\end{equation}}
\newcommand{\eey} {\end{eqnarray}}
\newcommand{\ovl}{\overline}
\newcommand{\vv}{\vspace{4mm}}
\newcommand{\lra}{\longrightarrow}
\newcommand{\SO}{\mathcal S}

\bibliographystyle{amsplain}

\section{Introduction}

Given a group endomorphism $\varphi:\pi \to \pi$, consider the (left) action of $\pi$ on $\pi$ via $\sigma \cdot \alpha \mapsto \sigma \alpha \varphi(\sigma)^{-1}$. The set of orbits of this action, denoted by $\mathcal R(\varphi)$, is the set of $\varphi$-twisted conjugacy classes or the set of {\it Reidemeister classes}. The cardinality of $\mathcal R(\varphi)$ is called the {\it Reidemeister number} $R(\varphi)$ of $\varphi$. The study of Reidemeister classes arises naturally in the classical Nielsen-Reidemeister fixed point theory (see e.g. \cite{J}). More precisely, for any selfmap $f:M\to M$ of a compact connected manifold $M$ with $\dim M\ge 3$, the minimum number of fixed points among all maps homotopic to $f$ is equal to the Nielsen number $N(f)$ which is bounded above by the Reidemeister number $R(f)=R(\varphi)$ where $\varphi$ is the induced homomorphism by $f$ on $\pi_1(M)$. While $N(f)$ is an important homotopy invariant, its computation is notoriously difficult. When $M$ is a Jiang-type space, then either $N(f)=0$ or $N(f)=R(f)$. While $N(f)$ is always finite, $R(f)$ need not be. Thus, when $R(f)=\infty$ we have $N(f)=0$ which implies that $f$ is deformable to be fixed point free. As a consequence of the $R_{\infty}$ property, it is shown in \cite{GW2} that for any $n\ge 5$, there exists a compact $n$-dimensional nilmanifold on which {\it every} self homeomorphism is isotopic to a fixed point free homeomorphism.

In \cite{LL}, it is shown that if $\varphi$ is an automorphism of a finitely generated non-elementary word hyperbolic group then $R(\varphi)=\infty$.
Since then many classes of groups have been shown to possess property $R_{\infty}$. However, most of the methods employed in these works have been ad hoc and specific to the classes of groups in question. On the other hand, $\Sigma$-theory, i.e., the Bieri-Neumann-Strebel invariant \cite{BNS}, has been used in \cite{GK} to prove property $R_{\infty}$ under certain conditions on $\Sigma^1$. Subsequent work in \cite{GS, GSS, KW2, SgW} further explore the use of $\Sigma$-theory in connection with property $R_{\infty}$. From the point of view of geometric group theory, it is natural to ask whether property $R_{\infty}$ is geometric, i.e., invariant up to quasi-isometry. In general, $R_{\infty}$ is not even invariant under commensurability and hence not invariant under quasi-isometry. The simplest example is that of $\mathbb Z$ as an index $2$ subgroup of the infinite dihedral group $D_{\infty}$ (see e.g. \cite{GW2}) where the former does not have $R_{\infty}$ while the latter does.

Since being non-elementary and word hyperbolic is geometric, the work of \cite{LL} implies that $R_{\infty}$ is invariant under quasi-isometry for the family of finitely generated non-elementary word hyperbolic groups (see also \cite{fel} in which a sketch of proof was given for non-elementary relative hyperbolic groups). Another family is that of the amenable or solvable Baumslag-Solitor groups $BS(1,n)$ for $n>1$. These groups have been completely classified in \cite{FM} up to quasi-isometry. For higher $BS(m,n)$ where $m\ge 2$ and $n>m$, it turns out that they are all quasi-isometric to each other as shown in \cite{Wh}. These Baumslag-Solitor groups (the fundamental group of the torus, $BS(1,1)$, is excluded here) have been shown in \cite{FeG} to have property $R_{\infty}$. More generally, the family of generalized Baumslag-Solitor (GBS) groups \cite{L} and any groups quasi-isometric to them also have property $R_{\infty}$ \cite{TW2}. Moreover, $R_{\infty}$ is also invariant under quasi-isometry for a certain solvable generalization of the $BS(1,n)$ \cite{TW1}. 

As another class of examples, let $\Lambda$ be an irreducible lattice in a connected semisimplie non-compact real Lie group $G$ with finite centre.  It is known that any finitely generated group $\Gamma$ quasi-isometric to $\Lambda$ has the $R_\infty$-property \cite{ms}.

Despite the success in \cite{LL,TW1,TW2,ms}, there have been no new examples of groups for which property $R_{\infty}$ is geometric. One difficulty is the determination of the group of quasi-isometries in general. As a first step, we ask

\begin{question}
For what class of groups is $R_{\infty}$ a commensurability property? Equivalently, if $G$ has property $R_{\infty}$ and $\Gamma$ is commensurable to $G$, (i.e., there exist subgroups $H < G, \bar H < \Gamma$ so that $H\cong \bar H, [G:H]<\infty, [\Gamma:\bar H]<\infty$) when does $\Gamma$ also have property $R_{\infty}$?
\end{question}

One of the main results of this paper is the following:   For the definition of the class of groups $\SO$ as well as the 
proof of the theorem, 
see \S 3.
\begin{theorem}\label{comm1}
Let $G$ be a finitely generated group. Suppose every finite index subgroup $H$ has the property that $b_1(H)=b_1(G)$. If $G\in \SO$ then every group $\hat G$ commensurable to $G$ also has property $R_{\infty}$.
\end{theorem}

The objective of this paper is to begin a systematic approach to studying $\Rinf$. We give conditions under which, when employing $\Sigma$ - theory, property $R_{\infty}$ is invariant under commensurability. In doing so, we introduce a stronger notion of $\Rinf$, namely $\Rcinf$ in section 2. When the complement $(\Sigma^1)^c$ is a finite spherical polytope lying inside an open hemisphere, we can find a point $[\chi]\in S(G)$ that is fixed by all automorphisms of $G$. If $[\chi]$ is {\it rigid} then $G$ has property $\Rcinf$ (Theorem \ref{SO-R}) and hence $\Rinf$. 

In section 4,  we investigate situations when the $\Sigma$ - invariants of $G$ are preserved under automorphisms of a finite index subgroup $H$. In section 5, we construct new families of groups that are direct products and free products with property $\Rinf$.

\section{Background on BNS invariants and $R_{\infty}$}

\subsection{Sigma invariants}

Let $G$ be a finitely generated group. The set ${\rm Hom}(G,\mathbb R)$ of homomorphisms from $G$ to the additive group $\mathbb R$ is a real vector space with dimension equal to $m$, the ${\mathbb Z}$-rank of the abelianization $G^{\rm ab}$ of $G$.  Denote by $\partial_{\infty}{\rm Hom}(G,\mathbb R)$ the boundary at infinity of $\mathbb R^m$ (ie.\ the set of geodesic rays in $\mathbb R^m$ initiating from the origin).  This is homeomorphic to the {\it character sphere of $G$} defined as the set of equivalence classes $S(G) := \{ [\chi] | \chi \in {\rm Hom}(G,\mathbb{R}) - \{ 0 \} \}$ where $\chi_1 \sim \chi_2$ if and only if $\chi_1 = r\chi_2$ for some $r > 0$.  Let $\Gamma$ denote the Cayley graph of $G$ with respect to a fixed generating set $S$.  Given $[\chi] \in S(G)$, define $\Gamma_\chi$ to be the subgraph of $\Gamma$ generated by the vertices $\{ g \in G \vert \chi(g) \geq 0 \}$.  We say $[\chi] \in \Sigma^1(G)$ if $\Gamma_{\chi}$ is path connected. 
For $n>1$, there are higher order $\Sigma$ - invariants $\Sigma^n$ introduced in \cite{BR}. 

The following are some well-known and useful facts (see e.g. \cite{Strebel}).   The notation $\Sigma^1(G)^c$ represents the complement of $\Sigma^1(G)$ in $S(G)$.

\begin{proposition}\label{Factor}
Suppose $\phi:G \to H$ is an epimorphism, and $\chi \in {\rm Hom}(H,\mathbb{R})$.  If $[\chi \circ \phi] \in \Sigma^1(G)$, then $[\chi] \in \Sigma^1(H)$. 
\end{proposition}

\begin{proposition}\label{Sigma-Product}
For finitely generated groups $G$ and $H$, $\Sigma^1(G \times H)^c = (\Sigma^1(G)^c \circledast \emptyset) \cup (\emptyset \circledast \Sigma^1(H)^c)$ where $\circledast$ denotes the spherical join on the character sphere $S(G\times H)$.
\end{proposition}

Consider a group extension given by the following short exact sequence
$$
1\to H\to G\to  K \to 1
$$
where $H$ and $G$ are finitely generated and $K$ is finite. Since $K$ is finite, the restriction homomorphism $\textrm{Hom}(G,\mathbb R)\to \textrm{Hom}(H,\mathbb R)$ is a monomorphism so that $S(G)$ can be regarded as a subsphere of $S(H)$. The following expression relates the $\Sigma$ - invariants of $G$ with those of $H$ (see \cite[Cor. 3.2]{KW2} or \cite[Thm. 9.3]{mmv}).

\begin{proposition}\label{sigma-finite}
For $n\ge 1$,
$$
\Sigma^n(G)=\Sigma^n(H)\cap \partial_{\infty}{\rm Fix} (\hat \nu)
$$
where $\nu: K\to G$ is any left transversal such that $\nu(1_K)=1_G$, and ${\rm Fix} (\hat \nu)=\{\phi \in {\rm Hom}(H,\mathbb R)\mid \phi(\nu(q)^{-1}h\nu(q))=\phi(h) \text{~for all $h\in H, q\in K$}\}$ is a subspace of ${\rm Hom}(H,\mathbb R)$ .
\end{proposition}

\subsection{Property $R^{\chi}_{\infty}$}

Recall from \cite{GK}, the role that $\Sigma$-theory plays is that the $\Sigma$-invariant can be used to obtain a rational point on the character sphere that is fixed by all automorphisms. It fact, the underlying principle is the existence of a character $\chi:G\to \mathbb R$ such that $\chi\circ \varphi=\chi$ for ALL $\varphi\in {\rm Aut}(G)$. In this case, the image ${\rm Im}(\chi)$ is a finitely generated abelian subgroup of $\mathbb R$ and is isomorphic to $\mathbb Z^r$ for some positive intger $r$. The equality $\chi\circ \varphi=\chi$ implies that $\varphi$ induces the identity on ${\rm Im}(\chi)$ which implies that $R(\varphi)=\infty$ since ${\rm Ker}(\chi)$ is characteristic.

\begin{definition}\label{rigid}  (Cf. \cite[Definition 1.6, \S6.E]{GSS}.)
Let $G$ be a finitely generated group and let $\chi: G\to \mathbb R$ be a non-trivial character. The character $\chi$ is said to be {\it rigid} if for any $r\in \mathbb R$, $r\cdot {\rm Im}(\chi)={\rm Im}(\chi)$ implies $r=\pm 1$. We say the character class $[\chi]$ is {\it rigid}
if for any $s>0$, the character $s\cdot \chi$ is rigid. 
\end{definition}

Note that a character $\chi$ is rigid if and only if its class $[\chi]$ is. 
Thus, if for all $\varphi \in {\rm Aut}(G)$, $[\chi \circ \varphi]=\varphi^*([\chi])=[\chi]$ and $[\chi]$ is rigid then $\chi \circ \varphi=\chi$ for all $\varphi\in {\rm Aut}(G)$. Evidently, if $[\chi]$ is rational (i.e., $\chi(G)\cong \mathbb Z$) then $[\chi]$ is rigid.  

Recall from \cite[\S6E]{GSS} that a character  $\chi$ as well as the class $[\chi]$ are called {\it transcendental} if 
$\im(\chi)\subset \mathbb R$ has the property that if $a, b\in {\rm Im}\chi$ are non-zero, then $a/b$ is either rational or transcendental.  It follows that if $[\chi]$ is transcendental then it is also rigid.  It is easily seen that if $\chi:G\to \mathbb R$ has image $\mathbb Z+2^{1/3} \mathbb Z$, then $[\chi]$ is rigid;
evidently it is not transcendental.

Suppose that $\chi:G\to \mathbb R$ is a transcendental character and $\chi':G\to \mathbb R$ is a character such 
that $Im(\chi')\subset Im(\chi)$.  Then $\chi'$ is also transcendental.  This is in general not true of 
rigid characters.  For example if $G=\mathbb Z^3, $ and $\chi,\chi' :G\to \mathbb R$ are defined as 
$\chi(a,b,c)=a+b\sqrt 2+c\pi, \chi'(a,b,c)=a+b\sqrt 2$, then $Im (\chi')\subset Im(\chi)$ but $\chi'$ 
is not rigid, even though $\chi$ is.

\begin{remark} Suppose a finitely generated group $G$ has a character sphere $S(G)$ of dimension $n=\dim S(G)$. Then for any automorphism $\varphi \in {\rm Aut}(G)$, the induced homeomorphism $\varphi^*: S(G)\to S(G)$ has topological degree $\pm 1$. The Lefschetz number $L(\varphi^*)=1+(-1)^n\cdot \deg \varphi^*$. Thus, if $n$ is even and $\deg \varphi^*=1$ then the Lefschetz Fixed Point Theorem asserts that $\varphi^*([\chi])=[\chi]$ for some $[\chi]$. However, there is no guarantee that $[\chi]$ is rigid. Similarly, if $\Sigma^1(G)^c$ is topologically a disk, then the Brouwer Fixed Point Theorem asserts every $\varphi^*$ has a fixed point but again such a fixed point need not be rigid. In fact, there exists a group $G$ \cite{GSS} where $S(G)$ has  a point $[\chi]$ that is fixed by $\varphi^*$ for all $\varphi\in {\rm Aut}(G)$ but $[\chi]$ is {\it not} rigid.
\end{remark}

The existence of such a globally {\it fixed} character {\it that is witnessed by $\Sigma$-theory} leads us to the following stronger notion of property $R_{\infty}$. 

\begin{definition}\label{witness}
A group $G$, not necessarily finitely generated, is said to have property $R^{\chi}_{\infty}$  if there exists a non-trivial character $\chi: G \to \mathbb R$ such that $\chi \circ \varphi=\chi$ for all $\varphi\in {\rm Aut}(G)$. Note that if $G$ has property $R^{\chi}_{\infty}$, it necessarily must have property $R_{\infty}$.
\end{definition}

\begin{example}\label{witness-ex}
Take $G=F_r\times BS(1,2)\times BS(1,2)$ where $F_r$ is the free group of rank $r\ge 2$. It is easy to see that the complement $[\Sigma^1(G)]^c=\mathbb S^{r-1}\cup \{+\infty\} \cup \{ +\infty'\}$ is an infinite set where $+\infty$ and $+\infty'$ denote the north poles of the two distinct copies of $S(BS(1,2))$ and $\mathbb S^{r-1}$ is a n
$(r-1)-$dimensional sphere disjoint from $+\infty$ and $+\infty'$. It follows that either each of the points $+\infty$ and $+\infty'$ is fixed, in which case, one of these endpoints yields a character that is fixed by all automorphisms, or $[\bar{ \chi}]$, which corresponds to the point on the arc obtained from taking the {\it average} of the characters $\chi, \chi \circ \varphi$ associated to those two points, is fixed by $\varphi^*$ for all $\varphi \in {\rm Aut}(G)$. Here $\varphi^*$ is the homeomorphism of $S(G)$ induced by $\varphi$. Since the points $+\infty$ and $+\infty'$ are rational, it follows that $[\bar \chi]$ is also rational and hence rigid. Again, we conclude that $\bar \chi$ is fixed by all automorphisms. Hence, $G$ has property $R^{\chi}_{\infty}$.
\end{example}

On the contrary, there are examples of groups for which the property $R_\infty$ does not imply the property 
$R_\infty ^\chi$.

\begin{example}\label{non-witness-ex}
By analyzing the automorphisms of the fundamental group of the Klein Bottle $K$ as in \cite[Lemma 2.1, Theorem 2.2]{GW2}, it is straightforward to see that there is no $[\chi]\in S(\pi_1(K))=\Sigma^1(\pi_1(K))\cong \mathbb S^0$ that is fixed by all automorphisms. Thus, $\pi_1(K)$ has property $R_{\infty}$ but not $R^{\chi}_{\infty}$.
\end{example}

\section{Conditions on the first Betti number}\label{Betti_1}

Consider a group extension
\begin{equation}\label{f-ext}
1\to H\to G\to K\to 1
\end{equation}
where $H$ and $G$ are finitely generated and $K$ is finite. Let $\nu :K\to G$ be a left transversal with 
$\nu(1_K)=1_G$.

The following simple relation between property $R_{\infty}$ for $G$ and that for $H$ is straightforward (see e.g. \cite{GW1}).

\begin{lemma}\label{Rinfty_for_finite}
Given the extension \eqref{f-ext},
if $H$ is characteristic and has property $R_{\infty}$ then $G$ has property $R_{\infty}$.
\end{lemma}

Recall that a subset $P \subset S=\mathbb R^n\setminus 0/\!\sim$ is a {\it spherical polytope} if there exist $v_1,\ldots,v_m\in 
\mathbb R^n$ such that (i) all the $v_j$ are in a half-space, that is, 
there exists a linear map $f:\mathbb R^n\to \mathbb R$ such that $f(v_j)>0~\forall j$, (ii)  
$P=
\{[v]\in S\mid v=\sum_{1\le j\le m} a_jv_j,~a_j\ge 0, v\ne 0\}$ and (iii) $v_1,v_2,\ldots, v_m$ 
 is minimal set of such 
vectors; equivalently no $v_i$ is a positive linear combination of a subset of $v_j, j\ne i$. 
Then $[v_1],\ldots,[v_m]$ are the {\it vertices} of $P$.   If $\alpha$ is any homeomorphism of $S$ induced by an 
automorphism of the vector space $\mathbb R^n$ such that $\alpha(P)=P,$ permutes the vertices of $P$.  

\begin{definition}\label{SO}
Let $\SO$ denote the class of all finitely generated groups which satisfy the following two conditions:
\begin{enumerate}
\item $\Sigma^1(G)^c$ is non-empty and 
lies inside an open hemisphere of the character sphere $S(G)$; 
\item the connected components of $\Sigma^1(G)^c$ are finite 
spherical polytopes with transcendental vertices.
\end{enumerate}
\end{definition}

Now let $G$ be any group, not necessarily finitely generated. Suppose that the abelianization $G^{\ab}\cong \mathbb Z^n\oplus T$ where $T\le G^{\ab}$ is the torsion subgroup.  Thus $G^\ab/T$ is free abelian of finite rank.     
One may still define the character sphere $S(G)$ exactly as for finitely generated groups and 
we note that it is homeomorphic to $S(G^{\ab})\cong \mathbb S^{n-1}$.  Moreover, if $\phi:G\to G$ is any automorphism, then $\phi$ induces 
a homeomorphism of $S(G)$.   A {\it  hemisphere} in $S(G)$ is defined as follows:  Suppose that $g\in G$ is not 
in $[G,G]$.  The hemispheres defined by $g$ are:
$H_g^+:=\{[\chi ]\in S(G)\mid \chi(g)>0\}$ and $H^-_g=\{[\chi]\in S(G)\mid \chi(g)<0\}$.
Clearly $H^\pm_g=H^\mp_{g^{-1}}$.  
The hemispheres are homeomorphic to the open balls in $S(G)$.  Moreover, 
if $\phi:G\to G$ is any automorphism, then $\phi^*:S(G)\to S(G)$ preserves the collection of hemispheres. 
These statements follow from the observation that any automorphism of $G$ induces an $\mathbb R$-linear automorphism of the ($n$-dimensional) vector space $\textrm{Hom}(G,\mathbb R)$.  In fact, automorphisms of $G$ preserve the $\mathbb Q$-structure 
$\textrm{Hom}(G;\mathbb Q)\subset \textrm{Hom}(G,\mathbb R)$.  It is readily seen that if $\chi$ is transcendental (resp. rigid), so is $\phi^*(\chi)=\chi\circ \phi$.

When $G$ is not finitely generated, the $\Sigma$-invariant $\Sigma^1(G)$ as in \cite{BNS} is not available. 
Although the definition due to K. S. Brown  \cite{Br} is applicable, we will not need it for our present purposes.

Let $K\lhd G$ be a {\it characteristic} subgroup of a group $G$ such that $\bar G:=G/K$ is a finitely 
generated infinite group.  We do not assume that $G$ is finitely generated.  
Then any automorphism $\phi$ of $G$ induces an automorphism $\bar \phi:\bar G\to \bar G$.  
This leads to a homomorphism $\Psi: \textrm{Aut}(G)\to \textrm{Aut}(\bar G)$.   Since $\bar G$ is finitely generated, one can apply $\Sigma$-theory to obtain $\Sigma^1(\bar G)\subset S(\bar G).$ 
Note that $\textrm{Aut}(G)$ acts on $S(\bar G)$ (via $\Psi$) preserving the sets $\Sigma^1(\bar G)$ and $\Sigma^1(\bar G)^c$ in $S(\bar G)$.  Observe that $\Sigma^1(\bar G)^c$ depends not only on $G$ but also 
on the choice of $K$ and is therefore not intrinsic to $G$.

\begin{definition}\label{snotfg} Let $\widetilde{\mathcal S}$ denote the class of all groups $G$ having a characteristic subgroup $K$ with quotient $G/K$ in $ \mathcal S$. 
\end{definition}

\begin{lemma}\label{average}
Let $G$ be any group such that $G^{\rm ab}$ has finite rank.
Suppose $\chi: G\to \mathbb R$ is a transcendental
character and $\phi:G\to G$ is an automorphism such that the
$\phi^*$-orbit of $[\chi]$ is finite. Let $\chi_j=\chi\circ \phi^j, 0\le j<r$
where $r>0$ is the least positive integer so that $[\chi_r]=[\chi]$.
Suppose that the $[\chi_j]$ are in an open hemisphere of $S(G)$.  Let
$\eta=\sum_{0\le j<r} \chi_j.$ Then $\eta$ is transcendental.\\
 \end{lemma}
\begin{proof} We note that ${\rm Im}(\eta)\subset {\rm Im}(\chi)$. So it suffices to
show that $\eta$ is non-zero.  But this follows from our hypothesis
that the $[\chi_j]$ are in an open hemisphere. \end{proof}

\begin{theorem}\label{SO-R}
If $G\in \widetilde{\mathcal S}$, then $G$ has property $R^{\chi}_{\infty}$.
\end{theorem}
\begin{proof} Let $K\lhd G$ be a characteristic subgroup such that $\bar G:=G/K$ is in $\mathcal S.$ 
As noted above, $\textrm{Aut}(G)$ acts on $S(\bar G)$ leaving $\Sigma^1(\bar G)^c$ invariant.   
Since $\bar G\in \mathcal S$, $\Sigma^1(\bar G)^c$ is a non-empty (finite) union of spherical polytopes.  Moreover, the group 
$\textrm{Aut}(G)$ acts on the 
set of 
vertices of $\Sigma^1(G)^c$. 
Since these vertices are transcendental and all of them are contained in an open half-space, by Lemma \ref{average}, we can find a transcendental character  $\chi: \bar G \to 
\mathbb R$ that is fixed by all automorphisms of $G$.  It follows that  the character $G\to \bar G\stackrel{\chi}{\longrightarrow} \mathbb R$ is fixed by $\textrm{Aut}(G)$.   Hence $G$ has property $R^{\chi}_{\infty}$.
\end{proof}

Denote by $b_1(\Gamma)$ the first Betti number of a group $\Gamma$.

\begin{lemma}\label{betti}
Given the extension \eqref{f-ext}, if $b_1(H)=b_1(G)$ and $G\in \SO$ then $H\in \SO$ and hence has property $R^{\chi}_{\infty}$.
\end{lemma}
\begin{proof} Since $b_1(H)=b_1(G)$, we conclude that the character sphere of $G$ coincides with the character sphere of $H$, that is, $S(G)=S(H)$. By \cite[Prop. 2.1, 2.3]{KW2}, 
$\partial_{\infty}{{\rm Fix} (\hat \nu)}=\partial_{\infty}{\rm Hom}(H,\mathbb R)$. It follows from Prop. \ref{sigma-finite} that $G$ and $H$ have the same $\Sigma$ invariants. Since $G\in \SO$ it follows that $H\in \SO$ and the last assertion follows from Theorem \ref{SO-R}.
\end{proof}

\begin{remark} It should be emphasized that if $G$ (and hence $H$ under the assumption $b_1(H)=b_1(G)$) has empty or symmetric (e.g. $\Sigma^1(\pi_1(M))=-\Sigma^1(\pi_1(M)$ where $M$ is a closed orientable $3$-manifold \cite{BNS}) $\Sigma$ - invariants then we simply cannot deduce any information regarding property $R_{\infty}$. For example, consider the classical lamplighter groups $L_n=\mathbb Z_n\wr \mathbb Z$. It is known \cite{GW1} that $L_n$ has property $R_{\infty}$ iff $\gcd (n,6)>1$. However, $\Sigma^1(L_n)=\emptyset$ for any $n\in \mathbb N$. Another such example is the fundamental group $\Gamma$ of a non-prime $3$-manifold where $\Gamma$ has property $R_{\infty}$ \cite{GSW2} but $\Sigma^1(\Gamma)=\emptyset$. Furthermore, if $M$ is a closed orientable $3$-manifold with $\mathbb H^2\times \mathbb R$ geometry then $\pi_1(M)$ has property $R_{\infty}$ \cite{GSW1} while the fundamental group of the $3$-torus does not. Here, both fundamental groups have non-empty symmetric $\Sigma^1$.
\end{remark}

We shall now prove Theorem \ref{comm1}. 

\begin{proof} Let $\hat G$ be commensurable to $G$ so that there exist $H\le G, \hat H\le \hat G$ such that $[G:H]<\infty, [\hat G:\hat H]<\infty$ and $\hat H\cong H$. Let $C_H$ be the core of $H$ in $G$ so that $C_H\le H$ and $C_H\unlhd G$. Since $H$ is of finite index in $G$ so is $C_H$. By Lemma \ref{betti}, we conclude that $C_H\in \SO$. Now $b_1(C_H)=b_1(H)=b_1(G)$. Furthermore, $H$ has the same $\Sigma$ - invariants as $G$ so we conclude that $H\in \SO$. Since $\hat H \cong H$, $\hat H\in \SO$. Now $\Gamma_{\hat H}:=\bigcap_{\varphi \in {\rm Aut}(\hat G)} \varphi(C_{\hat H})$ also has finite index in $\hat G$ and is characteristic in $\hat G$. Note that $\Gamma_{\hat H}$ is isomorphic to a subgroup $\bar H\le H$ of finite index in $H$. It follows from the assumption that $b_1(\bar H)=b_1(G)$, the subgroup $\bar H\in \SO$. Now $\Gamma_{\hat H}$ has property $R^{\chi}_{\infty}$. Applying Lemma \ref{Rinfty_for_finite}, we conclude that $\hat G$ has property $R_{\infty}$.
\end{proof} 

\begin{remark}
Lemma \ref{Rinfty_for_finite} does not necessarily imply $R^{\chi}_{\infty}$ for the extension unless it has the same $\Sigma$ - invariants as the kernel. Thus, in the proof of Theorem \ref{comm1}, if we know for instance that $b_1(\Gamma_{\hat H})=b_1(\hat G)$ then we can conclude that $\hat G$ also has property $R^{\chi}_{\infty}$.
\end{remark}

\begin{example}\label{BS}
Recall that property $R_{\infty}$ is a quasi-isometric invariant for the class of solvable Baumslag-Solitar groups (and their solvable  generalizations)\cite{TW1}. It is known (see e.g., \cite{KW2}) 
that $\Sigma^1(BS(1,n))=\{-\infty\}$ contains exactly one rational point and $b_1(BS(1,n))=1$. Furthermore, if $H$ is a finite index subgroup of $BS(1,n)$ then $H$ itself is a $BS(1,n^m))$ (see e.g. \cite{Bo}) so that $b_1(H)=1$. Thus Theorem \ref{comm1} gives a different proof of the fact that $R_{\infty}$ is invariant under commensurability for the class of solvable Baumslag-Solitor groups.
\end{example}

\begin{example}\label{gamma-n}
For any $n\ge 2$, write $n=p_1^{y_1}...p_r^{y_r}$ as its prime decomposition. Define a solvable generalization of the solvable Baumslag-Solitar groups by
$$\Gamma_n=\langle a,t_1,...,t_r \mid t_it_j=t_jt_i, t_iat_i^{-1}=a^{p_i^{y_i}}, i=1,...,r.\rangle.$$
Evidently, when $r=1$, $\Gamma_n=BS(1,n)$. In \cite{SgW}, it has been shown that $\Sigma^1(\Gamma_n)^c$ is a finite set of rational points all lying inside an open hemisphere so that $\Gamma_n\in \SO$. Moreover, a presentation is also found for any finite index subgroup $H$ of $\Gamma_n$. Using this presentation, one can show that $b_1(H)=b_1(\Gamma_n)=r$.
Thus Theorem \ref{comm1} gives a different proof of the fact that $R_{\infty}$ is invariant under commensurability for this class of generalized solvable Baumslag-Solitor groups.
\end{example}

Next, we exhibit more examples for which $b_1(H)=b_1(G)$. Note that in general, $b_1(G)\le b_1(H)$.

\begin{example}\label{semi-simple}
Let $G$ be a connected semi-simple Lie group
having real rank at least $2$ and $\Gamma$ be an irreducible lattice in $G$. For every finite index subgroup $H$ in $\Gamma$, $b_1(H)=b_1(\Gamma)=0$. However, in this case, $S(H)=\emptyset=S(\Gamma)$ and hence both $H$ and $\Gamma$ have empty $\Sigma$ - invariants.
\end{example}

\begin{example}\label{PL}
Certain subgroups of ${\rm PL}_{o}([0,1])$ (oriented PL-homeomorphism group of $[0,1]$) possess such property \cite[section 6]{GSS}.
\end{example}

\begin{example}\label{inner}
Suppose $G=H\rtimes_{\theta}K$ where $K$ is finite and 
$\theta:K\to {\rm Aut}(H)$ is the action. If $\theta(K)\subset {\rm Inn}(H)$ then $b_1(H)=b_1(G)$. From Stallings' $5$-term exact sequence, we have the following exact sequence
$$
H_2(K) \to H/[G,H] \to H_1(G) \to H_1(K)\to 0.
$$
Since $K$ is finite, both $H_2(K)$ and $H_1(K)$ are finite. It follows that 
$${\rm rk}_{\mathbb Z}\left(H/[G,H]\right)={\rm rk}_{\mathbb Z}(H_1(G))=b_1(G).
$$
Since $[H,H]\le [G,H]$, it suffices to show that $[H,H]=[G,H]$ under our assumptions. For any $g\in G$, $g$ can be uniquely written as $g=\hat h \bar k$ where $\bar k$ is the image of $k\in K$ under the section given by the splitting. For any $h\in H$,
\begin{equation*}
\begin{aligned}
ghg^{-1}h^{-1}&=\hat h\bar kh{\bar k}^{-1}{\hat h}^{-1}h^{-1} \\
                       &=\hat h \theta(k)(h){\hat h}^{-1}h^{-1} \\
                       &=\hat h \eta h\eta^{-1}{\hat h}^{-1}h^{-1} \qquad \text{for some $\eta\in H$ since $\theta(k)\in {\rm Inn}(H)$}\\
                       &=(\hat h\eta)h(\hat h\eta)^{-1}h^{-1} \in [H,H]
\end{aligned}
\end{equation*}
It follows that $[G,H]=[H,H]$ and hence we have $b_1(H)=b_1(G)$.
\end{example}

\begin{lemma}\label{simple_subgp}
Let $G$ be any group. Suppose the commutator subgroup $[G,G]$ contains an infinite simple group $K$ with $[[G,G]:K]<\infty$. Then for any finite index subgroup $H$ of $G$, $b_1(H)=b_1(G)$.
\end{lemma}
\begin{proof} To see this, first note that for every finite index subgroup $H$, its core $core_G(H)=C_H\le H$ is normal and has finite index in $G$. Now, $K\cap C_H$ has finite index in $K$ so $K\cap C_H$ is non-trivial. Since $K$ is simple and $core_K(K\cap C_H)\le K\cap C_H \le K$, it follows that $K\cap C_H=K$ so $K\le H$. Again, $K$ being simple means that $K=[K,K]$. Since $K=[K,K]\le [H,H] \le [G,G]$ and $K$ has finite index in $[G,G]$, we conclude that $[H,H]$ has finite index in $[G,G]$. It follows from Stallings' $5$-term exact sequence that $b_1(H)=b_1(G)$. Note that the argument above shows that {\it every} finite index subgroup of $G$ contains the simple group $K$. 
\end{proof}

\begin{remark}
The hypotheses of Lemma \ref{simple_subgp} are satisifed by a large class of groups. In particular, for $n\ge 2$, the Houghton groups $H_n$ satisfy the conditions of Lemma \ref{simple_subgp} with $K=A_{\infty}$.  Furthermore, let $S$ be a self-similar group and $G=V(S)$ be the associated Nekrashevych group. Then $[V(S), V(S)]$ is simple (see for instance \cite{N}). In fact, under certain conditions, $[G,G]$ can be of finite index in $G$ (\cite[Theorem 3.3]{SWZ}). Thus, by Lemma \ref{simple_subgp}, these aforementioned groups $G$ have the property that $b_1(H)=b_1(G)$  for all finite index subgroup $H$ in $G$. 
\end{remark}

The R. Thompson's group $F$ is known to have property $R_{\infty}$ \cite{BFG}. A different proof, using $\Sigma$-theory, has been given in \cite{GK}. In fact, one can conclude that $F\in \SO$ so $F$ has property $R^{\chi}_{\infty}$. Now, the next result follows from Lemma \ref{simple_subgp}, the fact that $[F,F]$ is simple and Theorem \ref{comm1} that any group commensurable to $F$ also has property $R_{\infty}$.

\begin{theorem}\label{Thompson_F}
Consider the R. Thompson's group $F$. Then any group commensurable to $F$ also has property $R_{\infty}$. 
\hfill $\Box$ 
\end{theorem}

\begin{remark}
Example \cite[5.5]{KW2} follows immediately from Theorem \ref{Thompson_F}. The generalized Thompson's groups $F_{0,n}$ have property $R_{\infty}$ and every group commensurable to one such also has property $R_{\infty}$. This result, including Theorem \ref{Thompson_F}, has been proven in \cite{GSS} using different methods.
\end{remark}

Another large class of interesting groups for which finite index subgroups have the same first Betti numbers is the class of lamplighter groups of the form $H\wr \mathbb Z$ where $H$ is a finite group.  Since lamplighter groups have empty $\Sigma^1$, these groups exhibit different behavior as we illustrate in the next example.

\begin{example}\label{strange-lamplighter}
Let $p\ge 5$ be an odd prime. It follows from \cite{GW1} that $G=\mathbb Z_p\wr \mathbb Z$ does not have property $R_{\infty}$.   Moreover, no finite index subgroup of $G$ has property $R_{\infty}$.  Since every subgroup of finite index in $G$ is of the form $(\mathbb Z_p)^k\wr \mathbb Z$ for some $k\in \mathbb N$, it follows from the main theorem of \cite{GW1} that such subgroup does not have property $R_{\infty}$. (See Remark 2.)
\end{example}

\section{Invariance under ${\rm Aut}(H)$}
Consider the Artin braid group $B_3$ (on the disk) and its pure braid group $P_3$ on $3$ strands. The group $P_3$ is a normal subgroup of index $6$ in $B_3$. Moreover, $P_3\cong F_2\times \mathbb Z$ where $F_2$ is the free group on $2$ generators and $\mathbb Z$ is generated by the central element $\Delta$ which is the full-twist of the $3$ strands. It follows that $b_1(P_3)=3$ and $b_1(B_3)=1$. By Prop. \ref{sigma-finite}, $\Sigma^1(B_3)=\Sigma^1(P_3)\cap \partial_{\infty} {\rm Fix} (\hat \nu)$. Since $[B_3,B_3]$ is finitely generated, $\Sigma^1(B_3)=\{\pm \infty\}$. Furthermore, $P_3$ has property $R_{\infty}$ (see e.g. \cite{FGW}).  Although $B_3$ and $P_3$ both have property $R_{\infty}$, neither of them belongs to $\SO$. Observe that every automorphism of $B_3$ restricts to an automorphism of $P_3$. This leads us to investigate when $\Sigma^n(G)$ is invariant under ${\rm Aut}(H)$.

Based on Prop. \ref{sigma-finite}, one should seek conditions under which $\partial_{\infty}{\rm Fix} (\hat \nu)$ is invariant under automorphisms of $H$. Recall that for any left transversal $\nu:K\to G$ such that $\nu(1_K)=1_G$,
$$
{\rm Fix} (\hat \nu)=\{\phi\in {\rm Hom}(H,\mathbb R)\mid \phi(\nu(q)^{-1}h\nu(q))=\phi(h), \forall h\in H, \forall q\in K\}.
$$
For every $q\in K$, define $\alpha_q\in {\rm Aut}(H)$ by $\alpha_q(h)=\nu(q)^{-1}h\nu(q)$. It follows that ${\rm Fix} (\hat \nu)=\{\phi\in {\rm Hom}(H,\mathbb R)\mid \phi \circ \alpha_q=\phi, \forall q\in K\}$. Denote by 
$\overline{\alpha_q}\in {\rm Out}(H)$ the image of $\alpha_q$ in ${\rm Out}(H)$.

\begin{proposition}\label{central-out}
Given a short exact sequence
$$
1\to H\to G\to K\to 1
$$
and a left transversal $\nu:K\to G$ with $\nu(1_K)=1_G$, if for every $q\in K$, $\overline{\alpha_q}\in Z({\rm Out}(H))$, the center of ${\rm Out}(H)$ then for any $\varphi \in {\rm Aut}(H)$, we have $\varphi(\Sigma^n(G))=\Sigma^n(G)$. Furthermore, if $G\in \SO$ and if $\varphi^*(S(G))=S(G)$ for some $\varphi \in {\rm Aut}(H)$  then $R(\varphi)={\infty}$.
\end{proposition}
\begin{proof}
Given any $\varphi\in  {\rm Aut}(H)$, there is an induced isomorphism $\hat \varphi$ on ${\rm Hom}(H,\mathbb R)$ given by $\hat \varphi(\phi)=\phi\circ \varphi$ for any $\phi\in {\rm Hom}(H,\mathbb R)$. Suppose $\phi \in {\rm Fix} (\hat \nu)$. For $\hat \varphi(\phi)\in {\rm Fix} (\hat \nu)$, we must have $\hat \varphi(\phi) \circ \alpha_q=\hat \varphi(\phi)$ for every $q\in K$. It follows that
$$
\phi\circ \varphi \circ \alpha_q=\phi\circ \varphi=\phi \circ \alpha_q\circ \varphi
$$
must hold for all $q\in K$. This equality holds if the automorphisms $\varphi \circ \alpha_q$ and $\alpha_q\circ \varphi$ differ by an inner automorphism. This holds under the assumption that $\overline{\alpha_q}$ lies in the center $Z({\rm Out}(H))$ for every $q\in K$. Now the invariance of $\Sigma^n(G)$ under ${\rm Aut}(H)$ follows from Prop. \ref{sigma-finite}. Since $G\in \SO$ there exists a rigid character $\chi$ that 
is fixed by all automorphisms of $G$. Since this character is obtained from the $\Sigma$ - invariants of $G$ which are invariant under ${\rm Aut}(H)$ and the subsphere $S(G)\subset S(H)$ is invariant under $\varphi$, we conclude that $\chi$ is also fixed by $\varphi^*$. It follows that $R(\varphi)={\infty}$.
\end{proof}

\section{More groups with $R^{\chi}_{\infty}$ or $R_{\infty}$}
Recall from Definition \ref{snotfg} the class of groups $\widetilde {\mathcal S}$.   A group $G$ is in $\widetilde{\mathcal S}$ 
if there exists a characteristic subgroup $K\triangleleft G$ such that $G/K$ belongs to $\mathcal S$.  It is not 
assumed that $G$ is finitely generated.   
 In this section  
we construct many families of groups (not necessarily finitely generated) that are direct products or free products of $G,H$ with $G\in \mathcal S$ where $H$ is a group belonging to certain families of groups described below.

(i)  {\bf Divisible groups.} Recall that a group $G$ is divisible if given any element $g\in G$ and an integer $n>1$, there exists an $h\in G$ 
such that $g=h^n$.  Examples of divisible abelian groups are $\mathbb Q^m\times (\mathbb Q/\mathbb Z)^n, m,n\in \mathbb N$.  
It is known that there exist $2^{\aleph_0}$-many pairwise non-isomorphic groups which are generated by two elements and divisible.  
(See \cite{lyndon-schupp}.)
These groups do not have any proper finite index subgroups. This family is closed under finite direct products.  We shall 
denote this class of group by $\mathcal D$. 

(ii) {\bf Torsion groups.} All torsion groups have vanishing $b_1$.  This follows easily from the basic fact that homology commutes 
with direct limit.  This family of groups is huge and includes many interesting groups such as Grigorchuk groups, the group of finitary permutations 
of $\mathbb N$, etc.   Elementary (abelian) 
examples include $A(\mathcal P):=\oplus_{p\in \mathcal P} \mathbb Z_p$, as $\mathcal P$ varies in the set of all (infinite) subsets of primes. 
Denote this class of groups by $\mathcal T$.

(iii) {\bf Acyclic groups.}  A group is said to be acyclic if its reduced homology with trivial $\mathbb Z$ coefficients vanishes.  
This class includes the Higman four-group \cite{dyer-vasquez}
and binate towers \cite{berrick}.  
It is known that  
any finitely generated group admits an embedding into a finitely generated acyclic group \cite{baumslag-heller-dyer}. The class of finitely generated acyclic groups, denoted $\mathcal A$, is closed under finite direct products and finite free products.  

(iv)   {\bf Higher rank lattices.}  Let $ G$ be a connected semisimple (real) linear Lie group which has no compact factors.  Suppose that 
the real rank of $G$ is at least $2$.  (The real rank of a linear Lie group is the dimension of the largest diagonalizable subgroup 
isomorphic to $\mathbb R_{>0}^\times$.)   Let $L\subset G$ be an irreducible lattice in $G$.  
Then it is a deep result of Margulis that any normal subgroup of $L$ is either finite or has finite index in $L$.  Since $L$ itself 
is not virtually abelian, it follows that $b_1(L)=0$ and that the same is true of any finite index subgroup of $L$.  (This is not true 
in the case of rank-$1$-lattices.)  Again, if $L_i\subset G_i, 1\le i\le n,$ 
are irreducible higher rank lattices, then the product $L:=\prod_{1\le i\le n} L_i $ also has trivial abelianization.  
Any finite index subgroup $\Lambda$ of $L$ admits a finite index subgroup $\Gamma$ which is a product $\prod \Gamma_i$ where each 
$\Gamma_i\subset L_i$ 
is a sublattice, i.e., finite index subgroup of $L_i$.  It follows that $b_1(\Gamma)=0$ and hence $b_1(\Lambda)=0$. Let us 
denote by $\mathcal L$ the class of all finite index subgroups $\Lambda$ of direct products $L=\prod_{1\le i\le n}L_i$ as above.

We now construct new examples of groups with property $R^{\chi}_{\infty}$.

\begin{proposition} \label{newgroups} We keep the above notation. Let $\mathcal C$ denote $\mathcal D\cup \mathcal T\cup  
\mathcal A\cup \mathcal L.$  Let $G$ be a group belonging to $\widetilde {\mathcal S}$ and  $H$ a group in $ \mathcal C$.  Suppose that
every homomorphism $H\to G$ is trivial.   Then:
(i)  $G\times H$ belongs to $\widetilde{\mathcal S}$ and hence has property $R^{\chi}_{\infty}$.
If $G$ is finitely generated, then $G*H\in \widetilde{ \SO}$ and so has property $R^\chi_\infty$. 
(ii) Let $K$ be a finite index subgroup of $G$ such that $b_1(K)=b_1(G)$.  Then $K\times H,  K*H \in \widetilde{\mathcal S}$ and so have property $R^{\chi}_{\infty}$.  
\end{proposition}
\begin{proof} 
(i) Since any homomorphism $H\to G$ is trivial, the subgroup $H=1\times H\subset G\times H$ is 
characteristic in $G\times H$.   Since $G\in \widetilde{\mathcal S}$, there exists a characteristic subgroup 
$K\triangleleft G$ such that $\bar G:=G/K$ is in $\mathcal S$.  Then $K\times H$ is characteristic 
in $G\times H$. This is because under an automorphism of $G\times H$, 
$H$ gets mapped onto itself and $K\mod H$ gets mapped isomorphically onto $K\mod H$.  Hence $G\times H\in \widetilde{\mathcal S}$. By Theorem \ref{SO}, $G\times H$ has property $R_\infty^\chi$.  

Next we consider $G*H$ where $G\in \widetilde {\mathcal S}$ is finitely generated.   
Let $K\triangleleft G$ be a characteristic subgroup of $G$ such that $\bar G:=G/K\in \mathcal S$. 
If $H\in \mathcal D \cup \mathcal T \cup \mathcal L$, then it is easy to see that $H$ is freely indecomposable.   If $H\in \mathcal A$ is acyclic then $H$ is {\it finitely generated}.  Thus, for any $H\in \mathcal C$, we 
have a free product decomposition $H=H_1* \cdots *H_n$ where each $H_j$ is indecomposable as a free product.  
Since $G$ is finitely generated, we also have 
$G=G_1*\cdots*G_m$ where each $G_j$ is indecomposable as a free product.  By our assumption, 
none of the $H_j$ are isomorphic to any of the  $G_i$.

A result of Fouxe-Rabinowitsch \cite{F-R}, describes a set of generators 
of $G*H=C_1*\cdots C_{m+n}$ where $C_i=G_i, i\le m, C_i=H_{i-m}, i>m$.  These are of three types, namely, (i) permutation automorphisms $\pi$ which permute the factors (having fixed once for all, isomorphisms between 
two factors $C_i\to C_j, i<j,$  if it exists); (ii){\it  factor automorphism} $\phi$ which maps each $C_i$ to itself,  and, 
(iii) $FR$-automorphisms $\sigma=\sigma(i,y), y\in C_j, j\ne i$ where $\sigma|C_i$ equals conjugation by $y$, $\sigma|C_k=id$ 
if $k\ne i$.  By our above observation, any permutation automorphism preserves the $H_j, 1\le j\le n.$ 
It follows that the normal subgroup  $N$ of $G*H$ generated by $H$ is invariant under each of these generators of $\textrm{Aut}(G*H)$.   So $N$ is characteristic in $G*H$.  It follows that any automorphism $\theta\in \textrm{Aut}(G*H)$ induces an automorphism $\theta_0\in \textrm{Aut}(G)$ via the natural projection 
$\eta: G*H\to G$.   Now $\theta_0\in \textrm{Aut}(G)$ induces an automorphism $\bar\theta\in \textrm{Aut}(\bar G)$.
Denoting by $q: G*H\to \bar G$ the composition of the natural quotients, we have $\bar \theta=q\circ \theta$.  
and ${\rm Ker}(q)$ is characteristic in $G*H$.   So $G*H\in \widetilde{\mathcal S}$ and so has the property $R_\infty^\chi$ by 
Theorem \ref{SO}.

(ii)  By Lemma \ref{betti}, $K$ belongs to $\SO$.  Since any homomorphism $H\to G$ is trivial, the same is 
true if $G$ is replaced by $K$.  Thus the hypotheses of the statement of the theorem are valid when $G$ is 
replaced by $K$.  Therefore (ii) follows from (i). 
\end{proof}

\begin{remark}
There are groups with property $R^{\chi}_{\infty}$ but with empty $\Sigma$ - invariants. For example, the group $BS(2,3)$ has property $R_{\infty}$. A close inspection of the proof in \cite{FeG} shows that $BS(2,3)$ has property $R^{\chi}_{\infty}$ while it has empty $\Sigma^1$ so that $BS(2,3) \notin \SO$.
\end{remark}

In general the requirement that any homomorphism $H\to G$ is trivial is hard to verify.  However, in certain contexts this 
is easily verified or known.   Examples of such situations are: (a) $H$ is a torsion group and $G$ is torsion free.  (b) $H$ admits 
no finite dimensional linear representation and $G$ is linear.  For example we may take $G$ to be an irreducible lattice in 
a semisimple linear Lie group and $H$ to be a binate group (\cite[Theorem 3.1]{berrick-1994},\cite{berrick-1995}).   (c) If $G$ is a group such that 
any nontrivial element in $G$ has at most finitely many roots in $G$ and $H$ is divisible. For example, take $G$ to be a non-elementary 
hyperbolic group or is a subgroup of $GL(n,\mathbb Z)$ for some $n$. Note that if $G$ is the fundamental group of a closed orientable hyperbolic $3$-manifold then by \cite{BNS}, $\Sigma^1(G)$ is symmetric so $G\notin \SO$. In view of this, (i) of  Proposition \ref{newgroups} can be generalized as follows using the same arguments as in the proof of Proposition \ref{newgroups}.

\begin{proposition} \label{newgroups2} Let $\mathcal C$ denote $\mathcal D\cup \mathcal T\cup  
\mathcal A\cup \mathcal L$. 
Let $G$ be a group with property $R_{\infty}$ and $H$ be in $ \mathcal C$.  Suppose that
every homomorphism $H\to G$ is trivial then:\\
(i)  $G\times H$ has property $R_{\infty}$, \\
(ii) if $G$ is freely indecomposable or is a finite free product 
then $G*H$ has property $R_{\infty}$. 
\end{proposition}
\begin{proof}
Following the proof of (i) of Proposition \ref{newgroups}, for (i), $H=1\times H$ is characteristic in $G\times H$ with quotient $G$. Since $G$ has property $R_{\infty}$, it follows (e.g. \cite[Lemma 1.2(1)]{GW1}) that $G\times H$ has property $R_{\infty}$. Similarly, for (ii), there is a characteristic subgroup $N$ in $G\ast H$ with quotient $G$ so that we can conclude that $G\ast H$ must also have property $R_{\infty}$.
\end{proof}

\subsection*{Concluding remarks}
Although the notion of $R^{\chi}_{\infty}$ is inspired by the use of $\Sigma$-theory,  
there are groups with such property but $\Sigma^1$ is empty (for instance, any free products have empty $\Sigma^1$). In the last section, we construct certain free products $G*H$ with property $R_{\infty}$. In particular, when $H\in \mathcal D$ is divisble, $H$ does not contain any proper subgroup of finite index. Yet, if $G$ has property $R_{\infty}$ (or $G\in \widetilde{\SO}$) and every $H\to G$ is trivial then $G*H$ has property $R_{\infty}$ (or $R^{\chi}_{\infty}$). On the other hand, it has been shown in \cite{GSW2} that $G*H$ has property $R_{\infty}$ provided both $G$ and $H$ are freely indecomposable and each contains proper characteristic finite index subgroups. We ask the following

\begin{question}
Let $G=G_1*\cdots *G_k$ be a finite free product of freely indecomposable (not necessarily finitely generated) groups $G_i$. Does $G$ necessarily have property $R_{\infty}$? 
\end{question}

\noindent
{\bf Acknowledgments:} We thank the referee of the for his/her careful reading of the paper and for his/her comments,  
which resulted in improved presentation.

\end{document}